\documentclass{article}
\usepackage[english]{babel}
\usepackage[utf8x]{inputenc}
\usepackage[T1]{fontenc}
\usepackage{amsmath,amssymb,amsthm,mathtools,mathdots,nicefrac,tikz}
\usepackage{mathrsfs}

\usepackage[a4paper,top=3cm,bottom=2cm,left=3cm,right=3cm,marginparwidth=1.75cm]{geometry}

\usepackage{amsmath}
\usepackage{ dsfont }
\usepackage{amsthm}
\usepackage{amssymb}
\usepackage{graphicx}
\usepackage{caption}
\usepackage[colorlinks=true, allcolors=blue]{hyperref}
\usepackage{subcaption}

\theoremstyle{definition} % non-italic theorems
\newtheorem{theorem}{Theorem}[section]

\newtheorem{corollary}[theorem]{Corollary}
\newtheorem{lemma}[theorem]{Lemma}
\newtheorem{proposition}[theorem]{Proposition}

\newtheorem{example}[theorem]{Example}
\newtheorem{remark}[theorem]{Remark}

\newtheorem{question}[theorem]{Question}

\newcommand{\s}{\mathcal{S}}

\newcommand{\sn}{\textrm{sn}}
\newcommand{\dsn}{\textrm{dsn}}
\newcommand{\scw}{\textrm{scw}}
\newcommand{\adh}{\textrm{adh}}

\DeclareMathOperator{\tw}{tw}
\DeclareMathOperator{\gon}{gon}

\begin{document}

\title{Graphs of scramble number two}
\author{Robin Eagleton and Ralph Morrison}
\date{}

\maketitle

\begin{abstract}
    The scramble number of a graph provides a lower bound for gonality and an upper bound for treewidth, making it a graph invariant of interest.  In this paper we study graphs of scramble number at most two, and give a classification of all such graphs with a finite list of forbidden topological minors.  We then prove that there exists no finite list of forbidden topological minors to characterize graphs of any fixed scramble number greater than two.

\end{abstract}

\section{Introduction}\label{section}
Chip firing games on graphs provide a combinatorial analog for divisor theory on algebraic curves \cite{bn}.  In chip firing games, a \emph{divisor} is a placement of  integer numbers of chips on the vertices of a graph.  A vertex can be then be fired to donate one chip along each incident edge, rearranging the chips into a new divisor. This brings us to the following question: how many chips do we need in a divisor to eliminate one arbitrarily placed negative chip from the graph?  We call this minimum number the \emph{gonality} of $G$.  Gonality has been extensively studied as a graph invariant and was proved to be NP-hard to compute in \cite{Gij}.  This motivates the finding of other graph invariants that bound gonaltiy and provide a simpler route to computing the gonality of a graph.  Scramble number, developed in \cite{Har}, is one such graph invariant that serves as a lower bound for gonality: $$\sn(G) \leq \gon(G).$$
One of the first theorems on scramble number was a complete characterization of the (connected) graphs of scramble number \(1\), which are precisely the trees \cite{Har}. In this paper we push further to characterize all graphs of scramble number \(2\). 
%While scramble number is also NP-hard to compute \cite{Ech}, it can still serve as a useful tool for studying gonality, .  Until now, there have been no structural characterizations for graphs of scramble number greater than one.  We seek to expand the structural characterizations of scramble number to two and higher, expanding our current knowledge of scramble number and gonality. 

\begin{figure}[hbt]
    \centering
\includegraphics{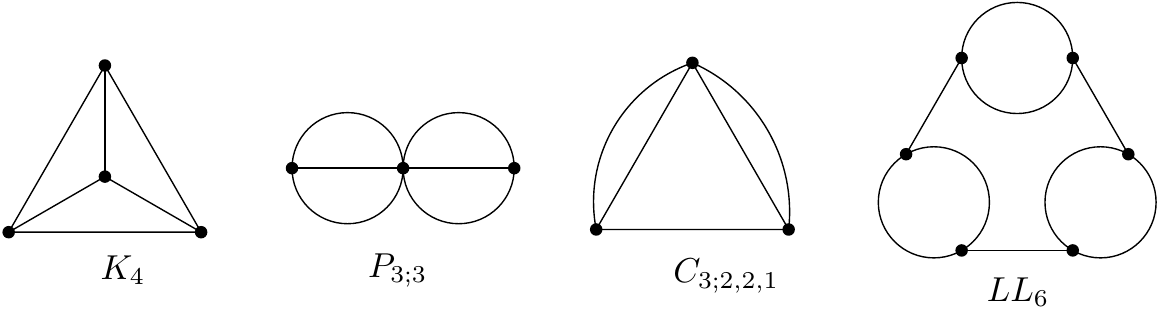}
    \caption{The four graphs from Theorem \ref{theorem:main_theorem}}
    \label{figure:four_forbidden}
\end{figure}

Our main theorem refers to the four graphs in Figure \ref{figure:four_forbidden}, which we denote by \(K_4\) (the complete graph on \(4\) vertices), \(P_{3;3}\) (a multipath), \(C_{3;2,2,1}\) (a multicycle), and \(LL_6\) (the loop-of-loops on \(6\) vertices).

\begin{theorem}\label{theorem:main_theorem}  A graph has scramble number at most \(2\) if and only if it has none of \(K_4\), \(P_{3;3}\), \(C_{3;2,2,1}\), and \(LL_6\) as a topological minor.
\end{theorem}

Since all graphs have a positive scramble number, and since the (connected) graphs of scramble number \(1\) are precisely the trees, a (connected) graph has scramble number exactly \(2\) if and only if it is not a tree and has none of the four graphs in Figure \ref{figure:four_forbidden} as a topological minor.

For our other primary result, let \(\mathscr{S}_m\) denote the set of all connected graphs of scramble number at most \(m\).

\begin{theorem}\label{theorem:non_result} The set \(\mathscr{S}_m\) admits a characterization by a finite list of forbidden topological minors if and only if \(m\leq 2\).
\end{theorem}

Our paper is organized as follows. In Section \ref{section:background} we present background material and useful lemmas.  In Section \ref{section:sn2} we prove Theorem \ref{theorem:main_theorem}.  In Section \ref{section:higher_orders} we construct for each \(k\geq 3\) an infinite families of graphs, topological minor minimal among those with scramble number \(k\), to prove Theorem \ref{theorem:non_result}.  In Section \ref{section:applications} we present some applications of our results.
    
    \smallskip

\noindent \textbf{Acknowledgements.} The authors thank Professor Colin Adams for suggestions and comments on an early draft of these results.  The authors were supported by NSF Grant DMS-2011743.

\section{Background and preliminaries}\label{section:background}

A \emph{graph} $G$ is a finite set of \emph{vertices}, $V(G)$, and a finite multiset of \emph{edges}, $E(G)$, such that every edge connects exactly two vertices. We allow for multiple edges to connect the same pair of vertices, but we do not allow an edge to connect a vertex to itself.  A graph is \emph{connected} if there exist a path of edges and vertices between every pair of vertices in $G$, and is \emph{disconnected} if it is not connected.   The \emph{edge connectivity} of a graph $G$, denoted $\lambda(G)$, is the minimum number of edges we must delete from $G$ to obtain a disconnected subgraph.  The edge connectivity of a graph offers more insight into the structure of the graph through the following integral graph theory theorem.
    
    \begin{theorem}[Menger's Theorem \cite{Men}]

    A graph $G$ has $\lambda(G)\geq k$ if and only if there exist $k$ edge-disjoint paths between every pair of vertices in $G$.
    \end{theorem}
%In Menger's Theorem, $k$ edge disjoint paths between two vertices refers to $k$ unique ways to travel along edges and vertices to get from one vertex to the other, allowing for these paths to use the same vertices. 

We define the \emph{degree} of a vertex $v$ to be the number of edges incident to $v$, i.e. the number of edges that have $v$ as an endpoint. A vertex \(u\) of degree two, incident to distinct vertices \(v\) and \(w\) via edges \(e_1\) and \(e_2\), can be ``smoothed over'', meaning we delete \(u\), \(e_1\), and \(e_2\) and add an edge between \(v\) and \(w\).  If a graph $H$ can be obtained from a graph $G$ by deleting vertices, deleting edges, and smoothing over vertices, $H$ is called a \emph{topological minor} of $G$, denoted $H \preceq G$. See Figure \ref{fig:extopomin} for an example.  %The concept of topological minors reappear frequently, and form the basis of our structural characterization of scramble number two. 

\begin{figure}[hbt]
    \centering
    \includegraphics[scale=1]{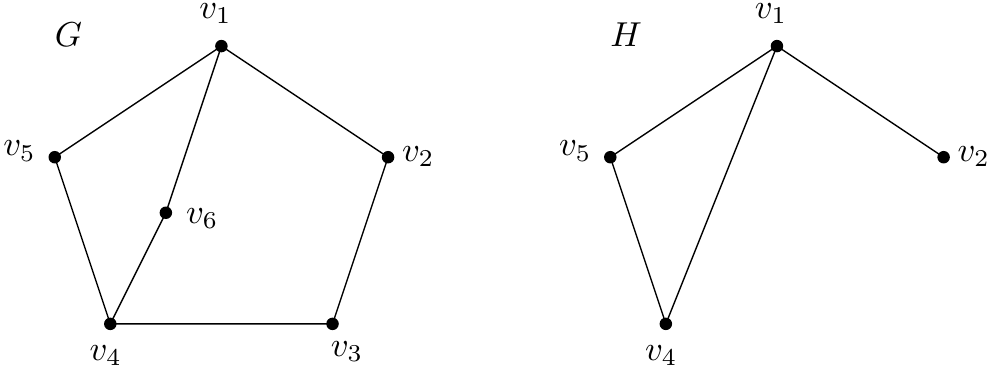}
    \caption{A topological minor \(H\) of a graph \(G\), obtained by deleting $v_5$ and its incident edges, and smoothing $v_6$.}
    \label{fig:extopomin}
\end{figure}

If \(H\) can be obtained from \(G\) simply by deleting edges and vertices, we call \(H\) and \emph{subgraph} of \(G\).  Given a subset \(S\subset V(G)\), we let \(G[S]\) denote the \emph{subgraph induced by \(S\)}; that is, the subgraph with vertex set \(S\) and edge multiset consisting of all edges from \(E(G)\) with both endpoints in \(S\).  If \(G[S]\) is a connected graph, we refer to \(S\) as a \emph{connected subset} of \(V(G)\).

Later on we will use work from \cite{reduction_method} applied to graphs of edge-connectivity \(3\). These results are most easily phrased in the language of \emph{pseudographs}, in which we do allow loops connecting vertices to themselves. A graph \(G\) is called \emph{\(3\)-edge-minimal} if \(\lambda(G)=3\), but for every edge in \(G\) we have \(\lambda(G-e)<3\).  We describe two operations on a graph \(G\):  in the operation $O_1^+$, an edge of \(G\) is subdivided and an edge is added to connect the new vertex to another vertex; and in the operation \(O^{(2)}_1\), an edge of \(G\) is subdivided to yield a new vertex \(z\), and then another edge (not adjacent to \(z\)) is subdivided to yield another new vertex \(w\), and then an edge is added to connect \(z\) and \(w\).

\begin{theorem}[Corollary 17 in \cite{reduction_method}, \(k=1\) case]\label{theorem:reduction_method}  Let \(G\) be a \(3\)-edge minimal multigraph with at least two vertices.  Then either there is a pseudograph \(G_1\) with \(\lambda(G_1)\geq 3\) and \(|V(G_1)|=|V(G)|-1\), from which \(G\) arises by \(O^+_1\); or else there is a pseudograph \(G_2\) with \(\lambda(G_2)\geq 3\) and \(|V(G_2)|=|G|-2\), from which \(G\) arises by \(O_1^{(2)}\).
\end{theorem}

We use this to prove the following corollary, which will be useful in the proof of Theorem \ref{theorem:main_theorem}.
\begin{corollary}\label{corollary:3_edge_top_minors}
    If \(G\) is a \(3\)-edge-connected graph on three or more vertices, then \(G\) contains one of \(K_4\), \(P_{3;3}\), and \(C_{3;2,2,1}\) as a topological minor.
\end{corollary}

\begin{proof} We prove this by induction on the number of vertices \(|V(G)|\).  If \(|V(G)|=3\), then the underlying simple graph of \(G\) is either the path on three vertices \(P_3\) or the cycle on three vertices \(C_3\).  In the first case, each edge must appear as \(3\) or more parallel copies to ensure \(3\)-edge-connectivity, so \(P_{3;3}\) is a subgraph.  In the second case, at least two edges must appear as \(2\) parallel copies to ensure \(3\)-edge-connectivity, so \(C_{3;2,2,1}\) is a subgraph.

Now let \(|V(G)|\geq 4\), and assume the claim holds for all \(3\)-edge-connected graphs with between \(3\) and \(|V(G)|-1\) vertices.  By Theorem \ref{theorem:reduction_method}, we can obtain \(G\) from a pseudograph \(G_1\) on one fewer vertex using \(O_1^+\), or from a pseudograph \(G_2\) on two fewer vertex using \(O_1^{(2)}\). 

We now deal with several cases.

\begin{itemize}
\item[(i)] The relevant \(G_i\) is a graph on \(3\) or more vertices.  By the inductive hypothesis, \(G_i\) has one of the three graphs as a topological minor; and \(G_i\) is itself a topological minor of \(G\), giving us the desired claim.
\item[(ii)] The relevant \(G_i\) is a graph on \(2\) vertices; this occurs only when \(|V(G)|=4\), and when \(G\) is obtained via \(O_1^{(2)}\) from \(P_{2;n}\) (a graph with two vertices connected by \(n\) edges) for some \(n\geq 3\), as these are all \(3\)-edge-connected graphs on \(2\)-vertices.  Up to symmetry there is only one way to perform \(O_1^{(2)}\) on \(P_{2;n}\), and it yields \(K_4\) as a subgraph.  Thus \(G\) has \(K_4\) as a topological minor.
\item[(iii)] The relevant \(G_i\) is a pseudograph that is not a graph, i.e. \(G_i\) has at least one loop.  Since \(G\) is a graph, the operation must eliminate any loops.  If \(G\) is obtained from \(G_1\) by \(O_1^+\), then the edge that is subdivided must be the loop, say \(l\) rooted at a vertex \(v\in V(G_1)\), where \(G_1-l\) is a simple graph.  We know that \(|V(G_1-l)|=|V(G_1)|=|V(G)|-1\), which is between \(3\) and \(|V(G)|-1\), and since removing \(l\) does not change edge-connectivity we know that \(\lambda(G_1-l)\geq 3\). This allows us to apply our inductive hypothesis to show that \(G_1-l\), and thus \(G\), has one of the three graphs as a topological minor.

If \(G\) is obtained from \(G_2\) by \(O_1^{(2)}\), there are several subcases to consider. The operation \(O_1^{(2)}\) must eliminate any loops in \(G_2\), of which there can be at most \(2\) by the structure of \(O_1^{(2)}\). Let \(L\) be the set of all loops in \(G_2\).  By the same logic as the previous argument, if \(|V(G)|\geq 5\) we have \(G_2-L\) is a \(3\)-edge-connected simple graph on \(3\) or more vertices, giving it (and \(G\)) one of the three graphs as a topological minors by the inductive hypothesis.  The last case to deal with is if \(|V(G)|=4\), and therefore \(|V(G_2)|=2\).  Since \(\lambda(G_2)\geq 3\), we know that \(G_2\) is of the form \(P_{2,n}\) with either \(1\) or \(2\) loops attached, possible on the same vertex or on different vertices.  If one loop, there are two ways to perform \(O_1^{(2)}\) to eliminate the loop, one of which yields \(P_{3;3}\) and the other of which yields \(C_{3;2,2,1}\) as a topological minor.  If two loops, there is a unique way to perform \(O_1^{(2)}\) to eliminate both of them; regardless of the placement of the loops, this yields \(C_{3;2,2,1}\) as a topological minor.  These operations are illustrated in Figure~\ref{figure:loopy_operations}.
\end{itemize}

\begin{figure}[hbt]
   \centering
    \includegraphics[scale=1]{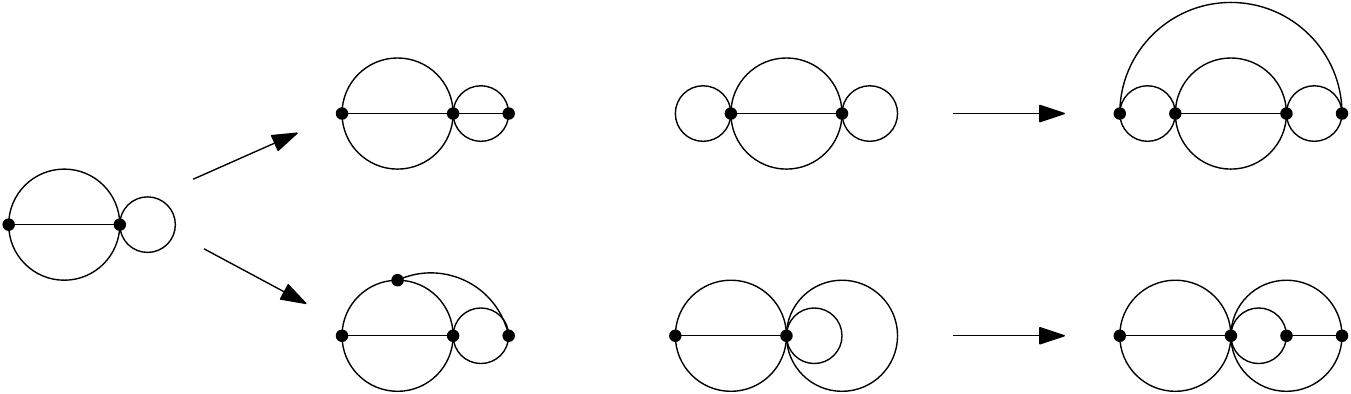}
    \caption{Possible operations to turn \(G_2\) into \(G\), one yielding \(P_{3;3}\) and the others yielding \(C_{3;2,2,1}\) as a toological minor}
    \label{figure:loopy_operations}
\end{figure}

We conclude by induction that every \(3\)-edge connected graph on three or more vertices has one of our three graphs as a topological minor. 
\end{proof}

We now move on to scramble number.  An \emph{egg} on a graph $G$ is a connected subset of vertices.  A \emph{scramble} on a graph $G$ is a collection of eggs on $G$.

Every scramble has an order, which requires several steps to calculate. An \emph{egg-cut} for a scramble is a collection of edges in \(E(G)\) that, when deleted, disconnect the graph into two components, each of which contains an egg.  
The \emph{egg-cut number} of a scramble $\s$, denoted $e(\s)$, is the minimum size of an egg-cut for \(\s\).  A \emph{hitting set} for a scramble is a set of vertices in \(V(G)\) such that every egg contains at least one vertex in that set.
The \emph{hitting number} of a scramble $\s$, denoted $h(\s)$, is the minimum size of a hitting set for $\s$.

These two definitions bring us to the \emph{order} of a scramble $\s$, which is defined as $$||\s||=\min\{h(S), e(S)\}.$$   The \emph{scramble number of a graph $G$} is then the maximum possible order of a scramble $\s$ on $G$.  That is,

$$\sn(G)=\max\limits_{\mathcal{S} \, \text{on}\, G}\{||\mathcal{S}||\}.$$

\begin{example}\label{example:sn_3}
Figure \ref{figure:four_forbidden_scrambles} presents a scramble on each of four graphs; in particular, the eggs are the circled collections of vertices.  Since for these examples all the eggs are disjoint, the hitting number for each scramble is the number of eggs, i.e. \(4\) for the scramble on \(K_4\) and \(3\) for the other three.  Each of the scrambles has an egg-cut number of \(3\); taking the minimum of the two relevant numbers, each scramble has order \(3\).  We remark that the egg-cut number may be smaller than the minimum number of edges incident to an egg; for instance, in the scramble on \(LL_6\), each egg is incident to $4$ edges, but an egg-cut of size \(3\) can be obtained by deleting two parallel edges and the edge opposite them in the underlying cycle graph \(C_6\).  Thus for any graph \(G\) in this figure we have \(\sn(G)\geq 3\); later, in Example \ref{example:scw_3}, we give an argument that \(\sn(G)=3\) for each.
\end{example}

\begin{figure}[hbt]
   \centering
    \includegraphics[scale=1]{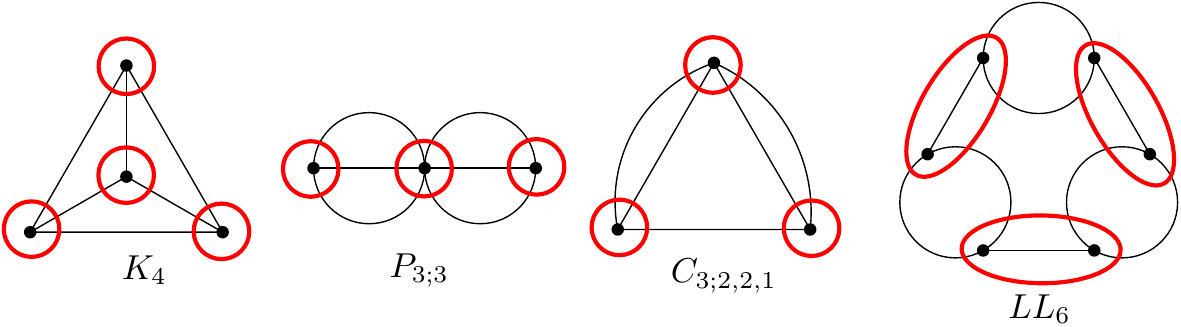}
    \caption{Example of scrambles of order three.}
    \label{figure:four_forbidden_scrambles}
\end{figure}

We now present some previous results.  To do so, we must first provide a definition of a tree. A graph $G$ is a  \emph{tree} if there is exactly one path between each pair of vertices of $G$.  Equivalently, a graph $G$ is a {tree} if $G$ is connected and contains no cycles.  

\begin{lemma}[Corollary 4.2 in \cite{Har}]\label{TREE}

A graph $G$ has $\sn(G)=1$ if and only if $G$ is a tree.
\end{lemma}

\begin{lemma}[Lemma 2.5 in \cite{Ech}]\label{edgeconnect}
Let $G$ be a graph.  Then, $\sn(G)\geq \min\{\lambda(G),|V(G)|\}$.
\end{lemma}

This is a powerful result for us as it is often easier to compute $\lambda(G)$ than trying to compute $\sn(G)$ directly.  For the next lemma, we define a \emph{bridge} to be an edge in $G$ such that deleting the edge disconnects $G$. %The resulting disconnected subgraphs of $G$ are called \emph{components}.

\begin{lemma}[Lemma 2.4 in \cite{Ech}]\label{bridge}
If $G$ is a graph with bridge $e$ and the two connected components of \(G-e\) are $G_1$ and $G_2$, then $\sn(G)= \max\{\sn(G_1), \sn(G_2)\}$.

\end{lemma}

These next two lemmas will allow us to conclude that scramble number is \emph{topological minor monotone}; that is, that scramble number can only decrease or remain unchanged when taking a topological minor.  This is an important step in using topological minors to study scramble numbers.  

\begin{lemma}[Proposition 4.5 in \cite{Har}]\label{lemma:subgraph}

If $H$ is a subgraph of $G$, $\sn(H)\leq \sn(G)$.

\end{lemma}

We remark that the same result does not hold if \(G'\) is a minor of \(G\); that is, it is possible for scramble number to increase when an edge is contracted \cite[Example 4.4]{Har}.

\begin{lemma}[Proposition 4.6 in \cite{Har}]\label{lemma:smoothing}

If $H$ is obtained from $G$ by smoothing vertices, $\sn(H) =\sn(G)$.
\end{lemma}

We can now provide a result on topological minors

\begin{corollary}\label{corollary:top_minor}
If $H$ is a topological minor of $G$, then $\sn(H) \leq \sn(G)$. 
\end{corollary}

\begin{proof}
If \(H\) is a topological minor of \(G\), then there exists a subgraph \(G'\) of \(G\) such that \(H\) is obtained from \(G'\) by smoothing at vertices.  By Lemma \ref{lemma:smoothing}, we have \(\sn(H)=\sn(G')\), and by Lemma \ref{lemma:subgraph}, we have \(\sn(G')\leq \sn(G)\). Thus \(\sn(H)\leq \sn(G)\).
 %Recall that a topological minor of a graph $G$ is obtained by deleting edges and vertices and smoothing vertices, and note that the order in which we perform these operations has no effect on the resulting graph.  Let $G'$ be a topological minor of $G$.  Now, let $H$ be a subgraph of $G$ such that $G'$ can be derived from $H$ by smoothing vertices. From Lemma \ref{subgraph} we have that $\sn(H) \leq \sn(G)$ and from Lemma \ref{smooth} we have that $\sn(H)=\sn(G')$.  Thus, $\sn(G') = \sn(H) \leq \sn(G)$, as desired.
\end{proof}

It follows that if we can find the scramble number of a topological minor of $G$, we have also found a lower bound on the scramble number of $G$ itself.  This process provides the underlying structure for proving Theorem \ref{theorem:main_theorem} in Section \ref{section:sn2}.

The final topic we recall from a previous paper is the \emph{screewidth} of a graph, introduced in \cite{screewidth}.  Given a graph \(G\), a \emph{tree-cut decomposition} is a pair \((T,\mathcal{X})\) such that \(T\) is a tree and \(\mathcal{X}\) is a set of subsets \(X_b\) of \(V(G)\), one for each \(b\in V(T)\), such that
\begin{itemize}
\item \(\bigcup_{b\in V(T)}X_b=V(G)\), and
\item \(X_{b_1}\cap X_{b_2}=\emptyset\) for \(b_1\neq b_2\).
\end{itemize}
That is, \(\mathcal{X}\) forms a \emph{near partition} of \(V(G)\), which is a partition with empty sets allowed.  For clarity, we will refer to the vertices and edges of \(T\) as \emph{nodes} and \emph{links}, respectively, reserving the terms \emph{vertices} and \emph{edges} for \(G\).

For \(l\in E(T)\), deleting \(l\) from \(T\) partitions the vertices of \(V(T)\), and thus the vertices of \(V(G)\), into two sets. The \emph{(link) adhesion of \(l\)}, denoted \(\adh(l)\), is the set of edges in \(E(G)\) connecting two vertices of \(G\) in these different sets.  Similarly, for \(b\in V(T)\) not a leaf, deleting \(b\) from \(V(T)\) partitions the vertices of \(V(T)\), and thus the vertices of \(V(G)\), into at least two sets.  The \emph{(node) adhesion of \(b\)}, denoted \(\adh(b)\), is the set of edges in \(E(G)\) connecting two vertices in these different sets. 

These adhesions admit an intuitive description.  Given a tree-cut decomposition of a graph \(G\), draw \(G\) in a thickened copy of \(T\) so that a vertex \(v\in V(G)\) is within the node \(b\in T\) such that \(v\in X_b\); and so that an edge connecting \(u,v\in V(G)\) is drawn along the unique path from \(b_1\) to \(b_2\) in \(T\), where \(u\in X_{b_1}\) and \(v\in X_{b_2}\).  The adhesion of a link \(l\) is then the number of edges from \(E(G)\) drawn passing through \(l\); and the adhesion of a node \(b\) is the number of edges from \(E(G)\) passing through \(b\) with neither endpoint in \(X_b\).

The \emph{width} of the tree-cut decomposition \(\mathcal{T}=(T,\mathcal{X})\) is defined to be the maximum of the following numbers:
\begin{itemize}
\item \(\max_{l\in E(T)}|\adh(l)|\);
\item \(\max_{b\in V(T)}|\adh(b)|+|X_b|\).
\end{itemize}
Finally, the \emph{screewidth} of a graph \(G\), denoted \(\scw(G)\), is the minimum possible width of a tree-cut decomposition of \(G\)

\begin{theorem}[Theorem 1.1 in \cite{screewidth}]\label{theorem:sn_scw} For any graph \(G\), we have \(\sn(G)\leq \scw(G)\).
\end{theorem}

\begin{figure}[hbt]
    \centering
\includegraphics[scale=0.75]{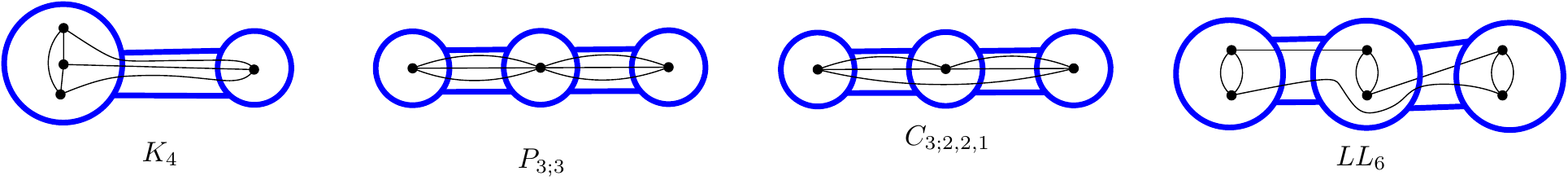}
    \caption{Tree-cut decompositions of width \(3\)}
    \label{figure:four_forbidden_tcd}
\end{figure}

\begin{example}\label{example:scw_3}
Figure \ref{figure:four_forbidden_tcd} illustrates a tree-cut decomposition for each of four graphs; the width of each decomposition is \(3\), so every graph \(G\) in the figure satisfies \(\scw(G)\leq 3\).  From Example \ref{example:sn_3}, we already had that each graph satisfied \(\sn(G)\geq 3\), which combined with Theorem \ref{theorem:sn_scw} gives us that \(\sn(G)=\scw(G)=3\) for all four graphs \(G\).

\begin{figure}[hbt]
    \centering
\includegraphics[scale=0.75]{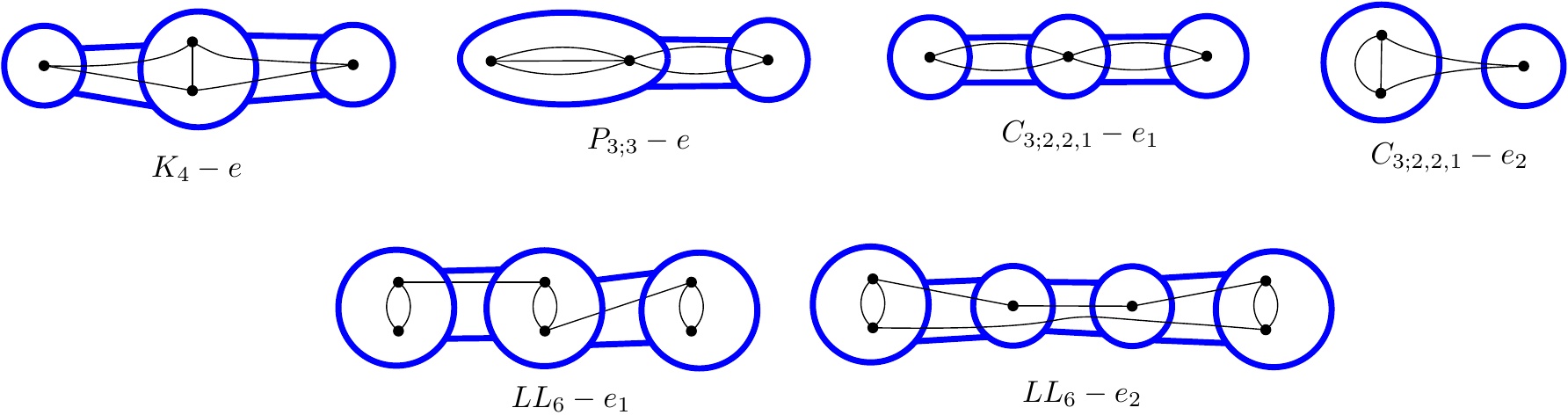}
    \caption{Tree-cut decompositions of width \(2\)}
\label{figure:four_forbidden_edges_deleted}
\end{figure}

In fact, we can say even more about these graphs: they are minimal, with respect to the topological minor relation, among all graphs of scramble number \(3\) or more. Since none has a degree \(2\) vertex, it suffices to show that deleting any edge will lower the scramble number.  Illustrated in Figure \ref{figure:four_forbidden_edges_deleted} are tree-cut decompositions of all subgraphs, up to symmetry, of these four graphs obtained by deleting a single edge.  Since each decomposition has width \(2\), each of these graphs \(H\) satisfies \(\sn(H)\leq \scw(H)\leq 2<3\), as desired.

\end{example}

\section{Characterizing Graphs of Scramble number Two}\label{section:sn2}

%In this section, we introduce a theorem giving a structural characterization of all graphs with scramble number two.  Beforehand, we define a new graph invariant relating to scramble number and prove a result that will be used in our proof.

We say a graph \(G\) is \emph{scramble minimal} if for any proper topological minor $H \preceq G$, $\sn(H)< \sn(G)$. We remark that since smoothing a vertex does not change scramble number, a graph is scramble minimal if and only if any degree \(2\) node is incident to two parallel edges (preventing a smoothing), and for every edge \(e\in E(G)\) we have \(\sn(G-e)<\sn(G)\).
A graph \(G\) is \emph{\(k\)-scramble minimal} if \(\sn(G)\geq k\) and any proper topological minor \(H\preceq G\) has \(\sn(H)<k\).

\begin{lemma}\label{lemlambda}
If $G$ is $k$-scramble minimal with $k$ or more vertices, then $\lambda(G) \leq k$. 
\end{lemma}

\begin{proof}
 Let $G$ be a $k$-scramble minimal graph, and suppose for the sake of contradiction that $\lambda(G) >k$. From the definition of $k$-scramble minimal we have that $\sn(G)=k$. Now delete any edge $e\in E(G)$ to create $G'$.  The graph $G'$ has $\lambda(G') \geq k$, so any egg-cut of $G'$ has size $k$ or greater.  From Lemma \ref{edgeconnect} we have that $\sn(G)\geq \min(\lambda(G),|V(G)|)$, so it follows that $\sn(G') \geq \min(k,k)$.  Thus, $\sn(G')\geq k$ which contradicts the $k$-scramble minimality of $G$.  Therefore, if $G$ is $k$-scramble minimal with $k$ or more vertices, $\lambda(G) \leq k$.
\end{proof}

Frequently in upcoming proofs it will be helpful to consider what happens to a scramble when a vertex or edge is deleted from a graph.  Because scrambles are defined to be specific to a graph, there is a priori no way of ``transferring'' a scramble to a subgraph.  To this end, for a scramble $\s$ on a graph $G$ with subgraph $H$, we construct $\s|_{H}$ on \(H\), which we refer to as \emph{the restriction of the scramble $\s$ to $H$}. For each egg $E \in \s$, if \(E\cap V(H)\) forms a nonempty connected subset in $H$, we let $E\cap V(H) \in \s|_{H}$.   Using this definition, we can extract the following result.

\begin{proposition}\label{ugh}
Let $\s$ be a scramble of order at least \(2\) on a graph $G$, let $e \in E(G)$ be an edge such that $H=G-e$ is connected, and let $\s'=\s|_{H}$.  Then, $||\s'|| \geq ||\s||-1$.

\begin{proof}
   Suppose for the sake of contradiction that this result does not hold. 
 Note that \(\s\) must contain at least one egg; if deleting \(e\) had disconnected every egg in \(\s\), then \(||\s||\leq h(\s)=1\). Now, we claim that $e(\s')\geq e(\s)-1$. If all eggs in \(\s'\) overlap, then $e(\s')=\infty$ and the claim holds.  Otherwise, choose an egg-cut \(T\) of size \(e(\s')\) for \(\s'\) on \(H\). 
 Then \(T\cup\{e\}\) is an egg-cut for \(\s\) on \(G\), separating the same pair of eggs as \(T\) did. Thus we have $e(\s')\geq e(\s)-1$.
  
  So, to have $||\s'|| < ||\s||-1$, we must have $h(\s')\leq h(\s)-2$. Let $S'$ be a minimum hitting set of $\s'$.  Letting \(u\) denote an endpoint of \(e\), note that $S'\cup \{u\}$ forms a hitting set of $\s$, since any egg of $\s$ that is not hit by $S'$ must have contained the edge $e$, and thus can be hit by the vertex $u$.  Therefore, we have found a hitting set of $\s$ of size  $h(\s')+1$, which contradicts $h(\s')\leq h(\s)-2$.  Hence it must be the case that $||\s'||\geq ||\s||-1$.
\end{proof}
\end{proposition}

\begin{corollary}
If $H=G-e$ for some non-bridge edge $e \in E(G)$, then $$\sn(G)-1 \leq \sn(H) \leq \sn(G).$$
\end{corollary}

\begin{proof}
Since $H$ is a subgraph of \(G\), we have by Lemma \ref{lemma:subgraph} that $\sn(H) \leq \sn(G)$.   Let $\s$ be a scramble on $G$ with $||\s||=\sn(G)$, and let $\s'=\s|_{H}$.  By Proposition \ref{ugh}, we have $||\s'||\geq ||\s||-1$, and so $\sn(H)\geq \sn(G)-1 $, as desired. 
\end{proof}

\begin{corollary}\label{corollary:minimal_equality}
Let \(k\geq 2\). If \(G\) is \(k\)-scramble minimal, then \(\sn(G)=k\).
\end{corollary}

\begin{proof}
By definition, if \(G\) is \(k\)-scramble minimal, \(\sn(G)\geq k\) and for any proper topological minor \(H\) we have \(\sn(H)<k\).  Since \(\sn(G)\geq 2\), \(G\) contains some edge such that \(H=G-e\) is connected.  By the previous corollary, \(\sn(G)-1\leq \sn(H)\leq k-1\).  Combined with \(\sn(G)\geq k\), it follows that \(\sn(G)=k\).
\end{proof}

We now have all of the necessary tools to characterize all graphs of scramble number two with our list of four forbidden topological minors. 

\begin{proof}[Proof of Theorem \ref{theorem:main_theorem}]
We know by Example \ref{example:scw_3} that all four graphs from Figure \ref{figure:four_forbidden} have scramble number \(3\).  By Corollary \ref{corollary:top_minor}, we therefore have that any graph with one of them as a topological minor has scramble number at least \(3\).  Contrapositively, a graph of scramble at most \(2\) has none of the four graphs as a topological minor.

%If \(G\) is a tree, then $\sn(G)=1$ from Lemma \ref{TREE} and so $\sn(G) \leq 2$ as desired.  Now suppose $G$ is a non-tree graph containing at least one of the four forbidden topological minors.  Because scramble number is topological minor monotone, $G$ must have scramble number greater than or equal to the scramble number of any of its topological minors.  So, it simply remains be shown that each of the four graphs in Figure \ref{fig:my_label} have scramble number three or greater to prove $\sn(G) \geq 3$.  First, note graph $A$ is isomorphic to $K_4$ (the complete graph on four vertices), which has already been proven to have $\sn(K_4)=3$ \cite{Ech}.  In graph $B$, let each of the three vertices be an egg of our scramble $\s$.  As each edge is a tripled parallel edge, we have $\lambda(B) =3$, and as $|V(B)|\geq 3,$ $\sn(B) \geq 3$ from Proposition \ref{edgeconnect}.  In graph $C$, again let each of the three vertices be an egg of our scramble $\s$.  As before, we have $\lambda(C) =3$ and thus $\sn(C) \geq 3$ by Proposition \ref{edgeconnect}.  For graph $D$, let the eggs of $\s$ be $e_1 =\{ v_1, v_6 \}$, $e_2 = \{ v_2, v_3\}$, and $e_3 = \{v_4, v_5 \}$.  We have that $|V(D)|\geq 3$ and to disconnect $D$ we must cut the cycle twice, and at least once at a doubled edge, giving us $\lambda(D)=3$ and thus $\sn(D)\geq 3$. Therefore, we have proved that if $G$ is a tree or $G$ contains one of the graphs in Figure $\ref{fig:my_label}$, $\sn(G) \neq 2$.
 
% \medskip

Let \(G\) be a graph of scramble number at least \(3\); we must show it has one of the four graphs as a topological minor.
Without loss of generality, we will assume \(G\) is \(3\)-scramble minimal; if it is not, then delete edges and vertices and smooth over degree \(2\) vertices until it is.  By Corollary \ref{corollary:minimal_equality}, \(\sn(G)=3\).  Note that $G$ must have at least three vertices as $\sn(G)= 3$.  %Let $\s$ be a scramble on on $M$ with $||\s||=3$.  
We have that $1 \leq \lambda(G) \leq 3$ from Lemma \ref{lemlambda}. However, if $\lambda(G)=1$, then $G$ has a bridge.  From Lemma \ref{bridge}, $G$ has a subgraph with the same scramble number, a contradiction to the minimality of \(G\). Thus, $ \lambda(G)> 1$.  If \(\lambda(G)=3\), then since \(|V(G)|\geq 3\) we may apply Corollary \ref{corollary:3_edge_top_minors} to conclude that \(G\) has one of \(K_4\), \(P_{3;3}\), or \(C_{3; 2,2,1}\) as a topological minor (indeed, \(G\) must equal one of these graphs, since \(G\) is \(3\)-scramble minimal).

It remains to handle the case of $\lambda(G)=2$.  If $\lambda(G)=2$, we know that there exist two edges $e_1,e_2 \in E(G)$ that form an edge-cut of $G$.  We will refer to the connected graphs obtained by deleting these edges as \(G_1\) and \(G_2\), with labels as illustrated in Figure \ref{figure:lambda2_part1}.  Note that a priori, it is possible that \(v_1=v_2\), and that \(v_3=v_4\).

\begin{figure}[hbt]
   \centering
    \includegraphics[scale=1]{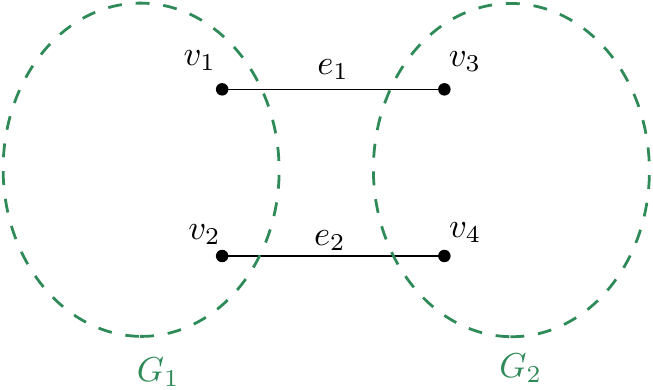}
    \caption{The structure of \(G\) with \(\lambda(G)=2\)}
    \label{figure:lambda2_part1}
\end{figure}

Let $\s$ be a scramble on $G$ with $||\s||=3$. As $\{e_1,e_2\}$ forms a $2$-edge cut of $G$, $\s$ cannot have one egg completely contained in $G_1$ and one egg completely contained in $G_2$.  Therefore, one subgraph (without loss of generality, \(G_1\)) does not completely contain any egg.  We now argue that \(v_1\neq v_2\)  Suppose for the sake of contradiction that  $v_1=v_2$. Let \(H\) be obtained from \(G\) by deleting all vertices in \(G_1\) besides \(v_1\). Then consider the scramble \(\mathcal{S}|_\mathcal{H}\) on \(H\).    Since no egg in \(\mathcal{S}\) is completely contained within $G_1$, and as all paths from $G_1$ to the rest of $G$ travel through $v_1$, the hitting number and egg-cut number remain unchanged.  Since \(G\) is minimal, it follows that \(G=H\).  Then, $v_1$ is a degree \(2\) vertex, which either contradicts the scramble minimality of $G$ (if \(v_1\) can be smoothed), or implies that \(v_3=v_4\), so that \(e_1\) and \(e_2\) are parallel edges connecting \(G_2\) to an isolated vertex \(v_2\).  In this case \(\mathcal{S}|_{G_2}\) has the same order as \(\mathcal{S}\), and we still have a contradiction to \(G\) being minimal. Thus we know that $v_1 \neq v_2$.

There must exist at least one path between $v_1 $ and $v_2$ within \(G_1\); otherwise $\lambda(G)=1$, with \(\{e_1\}\) forming an edge-cut.  If any two paths in \(G_1\) between $v_1$ and $v_2$ share an edge, we can use some shared edge $e$ and the fact that \(\{e,e_1\}\) forms an edge-cut to redefine our subgraphs $G_1$ and $G_2$ so that $G_1$ is smaller, and repeat with each remaining common edge. As we are losing vertices from $G_1$ this process must terminate.  Then, we have either multiple edge-disjoint path between the vertices of $G_1$; or our two edges forming an edge-cut come to one vertex in $G_1$, which would lead to a contradiction as before.  %Note that when redefining our $G_1$ subgraph we are not changing the structure or scramble number of $G$, just reducing the number of vertices in $G_1$.
Thus, we know there must be two edge disjoint paths between $v_1$ and $v_2$.  This gives us the graph illustrated in Figure \ref{figure:lambda2_part2} as a topological minor of \(G\).

\begin{figure}[hbt]
   \centering
    \includegraphics[scale=1]{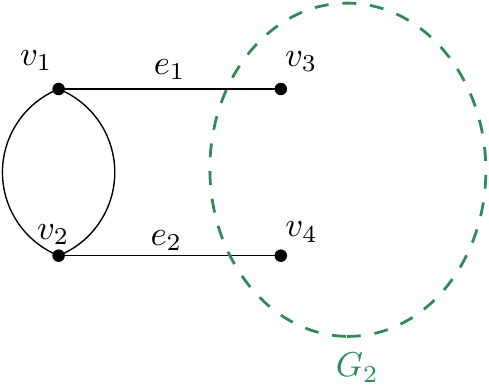}
    \caption{A topological minor of \(G\), which ends up equalling \(G\)}
    \label{figure:lambda2_part2}
\end{figure}

We claim that the graph \(H\) in  Figure \ref{figure:lambda2_part2} has scramble number equal to $3$ and therefore must actually be $G$ by the scramble minimality of \(G\).
Consider \(\mathcal{S}|_{H}\), the scramble obtained by intersecting the eggs of \(\mathcal{S}\) with \(V(H)\).  Eggs remain connected, as any egg whose connectivity relied on edges in \(G_1\) must contain both \(v_1\) and \(v_2\), which are connected in \(H\).  Any hitting set for \(\mathcal{S}|_{H}\) is also a hitting set for \(\mathcal{S}\), since every egg of \(\mathcal{S}\) intersecting \(G_2\) must contain \(v_1\) or \(v_2\); thus \(h(\mathcal{S}|_H)\geq h(\mathcal{S})\).  Now let \(E_1,E_2\in\mathcal{S}\), with \(E_i'=E_i\cap V(H)\in \mathcal{S}|_H\).  Suppose for the sake of contradiction that we may disconnect \(H\) by deleting a set \(T\) fewer than \(3\) edges so that \(E_1'\) is in one component and \(E_2'\) is in the other. Since \(\lambda(G)=2\), we know \(\lambda(H)=2\) as well, so \(|T|=2\).  Note that deleting the two edges connecting \(v_1\) and \(v_2\) does not disconnect the graph; nor does deleting one of those edges and any other edge \(e\) (otherwise \(\{e\}\)  would be an edge-cut of size \(1\)).  Thus \(T\subset E(G_2)\cup \{e_1,e_2\}\).  By the structure of our graphs, it follows that \(T\) is also an egg-cut for \(E_1\) and \(E_2\), contradicting \(e(\mathcal{S})\geq3 \).  Thus \(e(\mathcal{S}|_{H})\geq 3\), and we have \(||\mathcal{S}|_{H}||\geq 3\).  It follows that \(3\leq \sn(H)\leq \sn(G)=3\), so \(\sn(H)=3\).  As \(G\) is \(3\)-scramble minimal, we have \(G=H\).

Now, since $G$ is $3$-scramble minimal, deleting any edge decreases the scramble number of $G$ to two.  Consider $G'$, the subgraph of $G$ obtained by deleting one of the edges connecting \(v_1\) and \(v_2\).  No eggs of \(\mathcal{S}\)  are disconnected by this, so \(\mathcal{S}|_{G'}\) consists of the same eggs as \(\mathcal{S}\).  Since this scramble has the same hitting number, we know that  $\mathcal{S}|_{G'}$ must have an egg-cut of size at most two, since \(\sn(G')\leq 2\).  The graph $G$ does not have an egg-cut of size two, so the deleted edge must have been part of a minimal egg-cut of size $3$ in $G$; it follows that the other edge from \(v_1\) to \(v_2\) is part of the egg-cut in \(G'\), so in fact the egg-cut must have size exactly \(2\).  The remaining parallel edge is part of this cut, as well as some  edge in $G_2$, which must form a bridge of $G_2$.  Let this third egg-cut edge in $G_2$ be named $e_3$.  Now we can refine $G_2$ into two subgraphs, $G_A$ as the subgraph between $e_1$ and $e_3$ and $G_B$ as the subgraph between $e_2$ and $e_3$.  This structure is illustrated in Figure \ref{figure:lambda2_part3}.

\begin{figure}[hbt]
   \centering
    \includegraphics[scale=1]{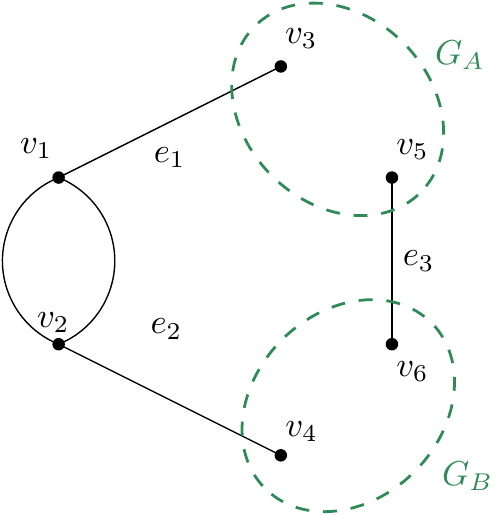}
    \caption{Further structure of \(G\)}
    \label{figure:lambda2_part3}
\end{figure}

Since \(\{v_1v_2,v_1v_2,e_3\}\) is an egg-cut, say separating the eggs \(E_1\) and \(E_2\), we claim without loss of generality that \(E_1\) contains the edge $e_1$ and \(E_2\) contains the  edge $e_2$.  If this were not the case, the parallel edges in $G_1$ could be replaced with one of $e_1$ or $e_2$, creating an egg-cut of size two for \(\mathcal{S}\), which is impossible. The eggs \(E_1\) and \(E_2\) do not intersect, so we know \(V(E_1) \subset V(G_A)\cup \{v_1\}\) and \(E_2 \subset V(G_B)\cup \{v_2\}\).  %As $h(\mathcal{S})\leq ||\mathcal{S}||= 3$, we know $\s$ must have another egg not overlapping with \(E_1\) or \(E_2\); otherwise \(\{v_1,v_2\}\) forms a hitting set of size two.  We claim that this egg, $E_3$, must contain the edge $e_3$.  If not, $e_3$ and either $e_1$ or $e_2$ would create an egg-cut of size two.  Additionally, $e_3$ cannot overlap with $e_1$ or $e_2$ without decreasing the hitting set.   The relative positions of $e_1$, $e_2$ and $e_3$ are illustrated below.

We now remark that any no egg in \(\s\) can be contained in \(V(G_A)\), as \(\{e_1,e_3\}\) would form an egg cut separating it from \(E_2\); and similarly no egg can be contained in \(G_B\).  We have already assumed that no egg can be contained in \(V(G_1)=\{v_1,v_2\}\).  Thus every egg contains at least one of \(e_1\), \(e_2\), and \(e_3\).  We note that some egg contains \(e_3\) without intersecting \(e_1\) or \(e_2\): otherwise \(\{v_3,v_4\}\) would be a hitting set of size \(2\). It follows that \(v_3\neq v_5\), and \(v_4\neq v_6\).  We now claim that there must exist two distinct (but not necessarily edge-disjoint) paths between \(v_3\) and \(v_5\) in \(G_A\).  Suppose not, so that there exists a unique path connecting them.  If some egg \(F_1\) containing \(e_1\) and some egg \(F_3\) containing \(e_3\) don't intersect, then \(e_2\) together with some edge on that path forms an egg-cut of size \(2\).  If every pair of such eggs intersects, then they must all intersect at a common vertex \(v\) of that path.  Since every other egg must contain the edge \(e_2\), we have that \(\{v,v_2\}\) is a hitting set of size \(2\), a contradiction.  Thus there are two distinct paths from \(v_3\) to \(v_5\) in \(G_A\), and by symmetry two distinct paths from \(v_4\) to \(v_6\) in \(G_B\).  This leads to \(G\) containing \(LL_6\) as a topological minor, illustrated below.  (In fact, since \(\sn(LL_6)=3\) and \(G\) is \(3\)-scramble minimal, we can conclude that \(G=LL_6\).)

%Now we examine the structure of $G_A$ and $G_B$, and without loss of generality consider $G_A$.  If the edges of $m$ and $p$ are not disjoint in $G_A$, meaning $v_3=v_5$, we claim we can reduce the size of $G$ without decreasing scramble number.  If there was an egg contained in $G_A$, it could disconnected from $e_2$ by deleting the two edges $m$ and $p$, a contradiction.  So, because there can be no eggs completely contained in $G_A$, we can delete $G_A-v_3$ and smooth $v_3$ without lowering scramble number.  It follows that $v_3 \neq v_5$.  We know there must be at least one path between $v_3$ and $v_5$ in $G_A$ to fulfil $\lambda(G)=2$.  If this is the only path between $v_3$ and $v_5$ in $G_A$, there must be some edge in $G_A$ but not in $e_1$ or $e_3$ that provides an egg-cut of size two along with the edge $n$.  Thus, there are at least two paths between $v_3 $ and $v_5$ in $G_A$ (these paths are not necessarily edge disjoint).  By symmetry the same is true for $v_4$ and $v_6 \in G_B$.  Note, this means that there is a cycle within $G_A$ and another in $G_B$.  The resulting structure has three loops in a cycle, which must contain graph $D$ as a topological minor.  As the above structure is contained within $G$, $G$ and therefore $G$ have graph $G_D$ as a topological minor.  Therefore any graph $G$ with $\sn(G) \geq 3$ contains one of our forbidden topological minors.

\begin{figure}[hbt]
   \centering
    \includegraphics[scale=1]{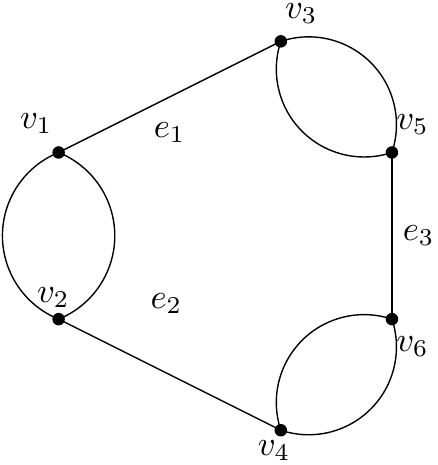}
    \caption{A topological minor of \(G\), which ends up equalling \(G\)}
    \label{figure:lambda2_part4}
\end{figure}

Thus in all cases, our \(3\)-scramble minimal graph \(G\) contains (in fact, is equal to) one of the four claimed topological minors.  This completes the proof.
\end{proof}

%This concludes our characterization of graphs with scramble number two.  From this point, the question arises: can we similarly characterize scrambles of higher orders?  We address this question and others in our next section.

\section{Scrambles of Higher Orders}
\label{section:higher_orders}

In this section we build families of graphs to prove that no finite forbidden topological minor characterization exists for \(\mathscr{S}_m\), the set of graphs of scramble number at most \(m\), when \(m\geq 3\). Our first lemma gives us a family of scramble minimal graphs with even scramble number.
%Here, we detail additional results we have found which expand upon our characterization of graphs with scramble number two.  The first result in this section resolves the question of generalizing Theorem \ref{theorem:main_theorem} to scrambles of higher orders.

\begin{lemma}\label{lemma:minimal_even} Let \(k\geq 2\) and \(n\geq 2k \), and let \(C_{n;k}\) denote the cycle graph on \(n\) vertices, where each edge has \(k\) parallel copies.  Then \(C_{n;k}\) is \(2k\)-scramble minimal.
\end{lemma}

\begin{proof}
Let \(\mathcal{S}\) denote the scramble whose eggs are the vertices of \(C_{n;k}\).  This scramble has order \(\min\{|V(C_{n;k})|,\lambda(C_{n;k})\}=\min\{n,2k\}=2k\), so \(\sn(C_{n;k})\geq 2k\).  See the left of Figure \ref{figure:even_scramble}.

To see that \(C_{n;k}\) is \(2k\)-scramble minimal, we must show that any proper topological minor of \(C_{n;k}\) has scramble number at most \(2k-1\). Since \(k\geq 2\), there are no degree \(2\) nodes.  Thus it suffices to show that deleting any edge gives a graph with scramble number at most \(2k-1\). Let \(v_1,\ldots,v_n\) denote the vertices of \(C_{n;k}\), ordered cyclically. By the symmetry of \(C_{n;k}\), we may consider \(H=C_{n;k}-e\), where \(e\) is an edge connecting \(v_1\) and \(v_n\).  Construct a tree-cut decomposition \(\mathcal{T}=(T,\mathcal{X})\) of \(H\) with \(T=P_n\), the path graph on \(n\) vertices, with nodes \(b_1,\ldots,b_n\) and \(X_{b_i}=\{v_i\}\) for all \(i\). See the right of Figure \ref{figure:even_scramble} for an example. The adhesion of the \(i^{th}\) link consists of \(2k-1\) edges, namely the \(k-1\) connecting \(v_1\) with \(v_n\) together with the \(k\) edges connecting \(v_i\) and \(v_{i+1}\).  Each node similarly has adhesion size \(k-1\), and the size of each \(X_b\) is \(1\).  Thus \(w(\mathcal{T})=2k-1\), so \(\sn(H)\leq \scw(H)\leq 2k-1\).  We conclude that \(C_{n;k}\) is \(2k\)-scramble minimal.
\end{proof}

\begin{figure}[hbt]
   \centering
    \includegraphics[scale=1]{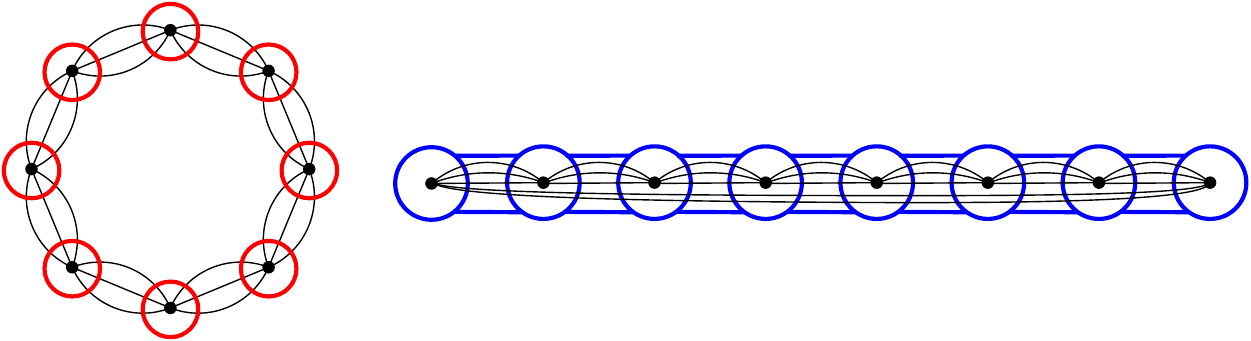}
    \caption{The graph \(C_{8;3}\) with a scramble of order \(6\), and a tree-cut decomposition of \(C_{8;3}-e\) of width \(5\)}
    \label{figure:even_scramble}
\end{figure}

Our next lemma gives us a family of scramble-minimal graphs with odd scramble number.

\begin{lemma}\label{lemma:minimal_odd} Let let \(k\geq 2\) and \(n\geq 3k\), and let \(\tilde{C}_{n;k}\) denote a cycle graph on \(n\) vertices where \(2k\) consecutive edges have \(k+1\) parallel copies (say from \(v_1\) through \(v_{2k+1})\), and all others have \(k\).  Then \(\tilde{C}_{n;k}\) is \((2k+1)\)-scramble minimal.
\end{lemma}

\begin{proof}
Consider the scramble
\[\mathcal{S}=\{\{v_2\},\ldots,\{v_{2k+1}\},\{v_{2k+2},\ldots,v_{n},v_1\}\}.\] An example of this scramble is illustrated in Figure \ref{figure:odd_scramble}.
Since the eggs are disjoint, we have \(h(\mathcal{S})=2k+1\).  For an egg-cut, at least two collections of parallel edges must be deleted, as this is necessary to disconnect the graph.  However, no two bundles of \(k\) edges can be deleted to disconnect two eggs; while one bundle of  \(k\) edges and another bundle of \(k+1\) can.  Thus \(e(\mathcal{S})=2k+1\), and so \(\sn(\tilde{C}_{n;k})\geq ||\mathcal{S}||=\min(n-k+1,2k+1)=2k+1\).

\begin{figure}[hbt]
   \centering
    \includegraphics[scale=1]{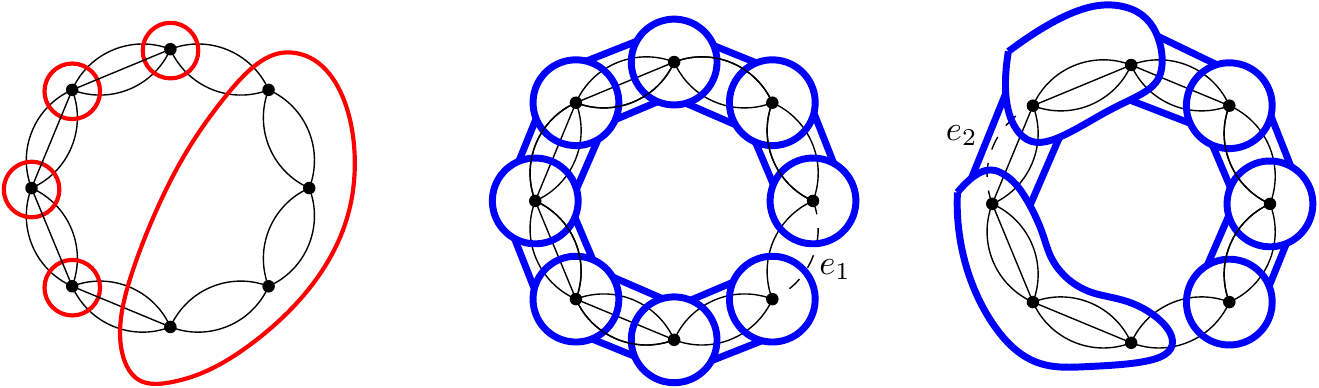}
    \caption{The graph \(\tilde{C}_{8;2}\) with a scramble of order \(5\), and tree-cut decompositions of \(\tilde{C}_{8;2}-e_1\) and \(\tilde{C}_{8;2}-e_2\), both of width \(4\)}
    \label{figure:odd_scramble}
\end{figure}

To see that \(\tilde{C}_{n;k}\) is \((2k+1)\)-scramble minimal, we first note that since \(k\geq 2\), there are no degree \(2\) nodes.  Thus it suffices to show that deleting any edge yields a graph with scramble number at most \(2k\).  First we deal with the case where the deleted edge \(e_1\) is from a bundle of \(k\) edges, say between \(v_i\) and \(v_{i+1}\) where \(2k+2\leq i\leq n\). 
 We construct a tree-cut decomposition similar to that from the proof of Lemma \ref{lemma:minimal_even}, except with the nodes in the path corresponding to \(v_{i+1},v_{i+2},\ldots,v_i\), ordered cyclically.  Then every (non-leaf) node and every link has as part of its adhesion the \(k-1\) edges connecting \(v_{i}\) and \(v_{i+1}\). See the middle of Figure \ref{figure:odd_scramble} for an example.  The nodes have no other adhesion (and correspond to sets of size \(1\)), and every link has either \(k\) or \(k+1\) more edges in its adhesion.  Thus \(\sn(\tilde{C}_{n;k}-e_1)\leq \scw(\tilde{C}_{n;k}-e_1)\leq 2k\).   
 
 Now we handle the case where the deleted edge \(e_2\) came from a bundle of \(k+1\) nodes, say between \(v_i\) and \(v_{i+1}\) where \(1\leq i\leq 2k+1\); by the symmetry of the graph, we may assume \(k+1\leq i\leq 2k+1\).  Construct a tree-cut decomposition with \(T=P_{n-i+1}\) on nodes \(b_1,\ldots,b_{n-2}\) where \(X_{b_1}=\{v_1,\ldots,v_{i}\}\), \(X_{b_2}=\{v_{i+1},\ldots,v_{2k+2}\}\), and \(X_{b_j}=\{v_{2k+j}\}\) for \(3\leq j\leq n-2k\).  See the right of Figure \ref{figure:odd_scramble} for an example. By our choice of \(i\), \(|X_{b_1}|\leq 2k\), and as \(b_i\) is a leaf node it has no adhesion.  Similarly, \(|X_{b_2}|\leq k\), and its adhesion consists of the \(k\) edges connecting \(v_1\) and \(v_n\).  Every other \(b_i\) has \(|X_{b_i}|=1\), with the non-leaves having the same adhesion of \(k\) edges. Finally, every link has size \(k+k=2k\).  Thus \(\sn(\tilde{C}_{n;k}-e_2)\leq \scw(\tilde{C}_{n;k}-e_2)\leq 2k\). 
We conclude that  \(\tilde{C}_{n;k}\) is \((2k+1)\)-scramble minimal.
\end{proof}

We are ready to prove that  \(\mathscr{S}_m\), the set of all connected graphs of scramble number at most \(m\), admits a characterization by a finite list of forbidden topological minors if and only if \(m\leq 2\).

\begin{proof}[Proof of Theorem \ref{theorem:non_result}]
Note that \(\mathscr{S}_1\) is precisely the set of trees.  A connected graph is a tree if and only if it does not contain a cycle as a subgraph if and only if it does not have as a topological minor the graph \(P_{2;2}\), consisting of a pair of vertices connected by a pair of edges.  Thus  \(\mathscr{S}_1\) admits a characterization by a finite list of forbidden topological minors.  We also have such a characterization for \(\mathscr{S}_2\) given by Theorem \ref{theorem:main_theorem}.

Now let \(m\geq 3\).  Any complete list of forbidden topological minors for \(\mathscr{S}_m\) must include all \((m+1)\)-scramble minimal graphs.  By Lemmas \ref{lemma:minimal_even} and \ref{lemma:minimal_odd}, there are infinitely many \((m+1)\)-scramble minimal graphs, namely \(C_{n;(m+1)/2}\) for \(m+1\) even and \(\tilde{C}_{n,\lfloor (m+1)/2\rfloor}\) for \(m+1\) odd, as \(n\) varies above the prescribed minimum.  Thus there can exist no characterization of \(\mathscr{S}_m\) by a finite list of forbidden topological minors.
\end{proof}

Although we have shown that there exists no finite list of forbidden topological minors for \(\mathscr{S}_m\) with \(m\geq 3\), we might still hope for some nice characterization.  For instance, perhaps we could classify all \((m+1)\)-scramble minimal graphs as nicely described infinite families, with a finite collection of special cases.  Another possibility comes from generalizing the notion of a topological minor, an approach we detail here.

Let \(u\) be a vertex in a graph \(G\) adjacent to exactly two other vertices, say \(v\) by \(m\) edges and \(w\) by \(n\) edges.  A \emph{multi-smoothing at \(u\)} removes \(u\) from the graph and adds \(\min\{m,n\}\) edges between \(v\) and \(w\).  We say that a graph \(H\) is a \emph{multi-topological minor} of a graph \(G\) if it can be obtained from \(G\) by deleting vertices and edges and performing multi-smoothings.  For example, the graph \(C_{n;k}\) is a multi-topological minor or \(C_{n+1;k}\), obtained by performing a multi-smoothing at any vertex.

\begin{proposition}
If \(H\) is a multi-topological minor of \(G\), then \(\sn(H)\leq \sn(G)\).
\end{proposition}

\begin{figure}[hbt]
    \centering
    \includegraphics[scale=1]{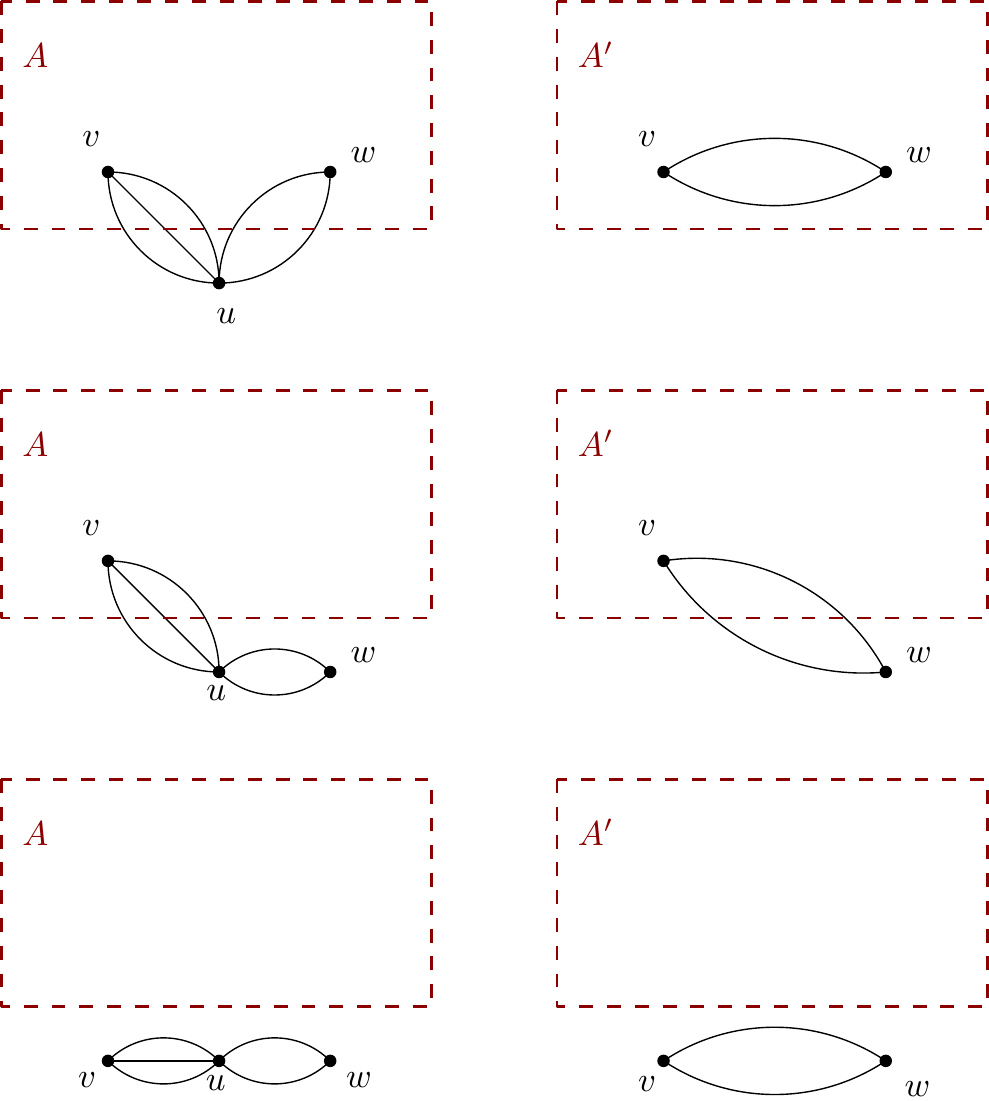}
    \caption{Possible relations of \(u,v,w\) to an egg-cut during a multi-smoothing}
    \label{figure:multi_smoothing_cases}
\end{figure}

\begin{proof}
It suffices to assume that $H$ can be obtained from $G$ by multi-smoothing a single vertex, say $u$ with adjacent vertices $v$ (connected by \(m\) edges) and $w$ (connected by \(m\) edges).  After smoothing $u$, the number of edges between $v$ and $w$ is $\min\{m,n\}$. Let \(e'\) be one such edge.

Now, let $\s'$ be some scramble on $H$.  We will define the scramble $\s$ on $G$ as follows.  For each egg $E_i'\in \s'$, include $E_i \in \s$, where \begin{equation}
E_i = \left\{
    \begin{array}{lr}
        E_i' & \text{if } \, v \notin e'\\
        E_i' \cup \{u\} & \text{if } \, v \in e'.
    \end{array}
\right. 
\end{equation} Consider a minimal hitting set \(S\) for $\s$.  If $u \notin S$, then $S$ must also be a hitting set for $\s'$.  If $u \in H$, then $H'=(H \cup \{v\}) - \{u\}$ is a hitting set for \(\s'\).  It follows that $h(\s')\leq h(\s)$. %We know that $H'$ is a hitting set of $\s'$ as every egg of $\s'$ is either directly equivalent to an egg of $\s$ or contains the vertex $v$.  So, if an egg $E' \in \s'$ is not hit by a vertex of $H$, $E'$ must be hit by $v$ and thus $H'$ is a hitting set of $\s'$.  As $|H|=|H'|$, we have that $h(\s')\leq h(\s)$.  

Now consider a set $A \subseteq V(G)$ such that $A$ and $A^c$ both contain eggs of $\s$.  Without loss of generality, $u \notin A$. It follows that $A' \subseteq V(G')$ and $(A')^c$ both contain eggs of $\s'$ in $G'$. Then, we have the three possible situations, illustrated in Figure \ref{figure:multi_smoothing_cases}.  We could have $v,w \in A$, one of $v$ or $w$ in $A$, or both $v,w \notin A$.  We claim that in passing from \(G\) to \(H\), the number of edges in our egg-cut could only decrease.  If \(v,w\in A\), then we lose \(m+n\); if \(v\in A\) and \(w\notin A\) we lose \(m-\min\{m,n\}\) (and similar if \(w\in A\) and \(v\notin A\)); and if \(v,w\notin A\) we lose no edges. Thus for any egg-cut for \(\s\), there is an egg-cut with at most that many edges for \(\s'\), meaning \(e(\s')\leq e(\s)\).

It follows that \(||\s'||\leq ||\s||\), implying that \(\sn(H)\leq \sn(G)\), as desired.
%Note, the edges between $A$ and $A^c$ form an egg-cut of $\s$ and the edges between $A'$ and $A'^c$ form an egg-cut of $\s'$.  Because $A$ is an arbitrary set of vertices not containing $u$, $A$ could be any egg-cut of $\s$ and similarly $A'$ could be any egg-cut of $\s'$.  In each above case, the number of edges between $A$ and $A^c$ is less than or equal to the number of edges between $A'$ and $A'^c$.  Thus, $e(\s') \leq e(\s)$ and so we have $|\s'| \leq |\s|$ for all scrambles $\s'$ on $G'$.  Therefore, if $G'$ is obtained from $G$ by smoothing one vertex, $\sn(G') \leq \sn(G)$.   
\end{proof}

Since scramble number is multi-topological minor monotone, we could hope for a finite forbidden multi-topological minor characterizationof \(\mathscr{S}_m\) for \(m\geq 3\).  For instance, for \(m+1\) even and \(n\geq 2(m+1)\), the graphs \(C_{n,(m+1)/2}\) all have \(C_{2(m+1),(m+1)/2}\) as a multi-topological minor. Moreover, \(C_{2(m+1),(m+1)/2}\) is minimal among graphs of scramble number \(2(m+1)\) with respect to the multi-topological minor relation: performing any multi-smoothing would decrease the number of vertices to \(2m+1\), bringin the scramble number down.  For \(m+1\) odd, the graph \(\tilde{C}_{2(m+1)+1; \lfloor (m+1)/2\rfloor}\) plays a similar role.  An interesting direction for future work would be to determine if this multi-topological minor relation could lead to finite characterizations of \(\mathscr{S}_m\).

\section{Applications}\label{section:applications}

In this section we present several applications of Theorem \ref{theorem:main_theorem}.

The first result we present is on a variant of scramble number. We say that a scramble \(\s\) on a graph \(G\) is \emph{disjoint} if \(E_i\cap E_j=\emptyset\) for all distinct eggs \(E_i,E_j\in\s\).  The \emph{disjoint scramble number of a graph} $G$, denoted $\dsn(G)$, is the maximum possible order a disjoint scramble on \(G\).
Note that every disjoint scramble on a graph $G$ is also a scramble on a graph $G$, so $\dsn(G) \leq \sn(G)$. 
 In practice, disjoint scrambles are  easier to work with, since the computation of \(h(\mathcal{S})\) is immediate (in particular, it equals the number of eggs).  As we will see in the following example, we may have \(\dsn(G)\neq \sn(G)\).

 \begin{example}\label{example:w5}
 Consider the wheel graph \(W_5\), consisting of a cycle on \(5\) vertices together with a vertex connected to all other vertices (pictured on the left in Figure \ref{figure:wheel_graph_scw}).  We claim that \(\sn(W_5)=4\).  Pictured in the middle of Figure \ref{figure:wheel_graph_scw} is a tree-cut decomposition of \(W_5\) of width \(4\), so \(\sn(G)\leq \scw(G)\leq 4\).  For the lower bound, we consider the \(2\)-uniform scramble \(\mathcal{E}_2\) on \(W_5\), whose eggs are all sets of the form \(\{u,v\}\) with \(u\neq v\) and \(W_5[{u,v}]\) connected (that is, with \(uv\in E(W_5)\)).  By \cite[Theorem 3.1]{uniform_scrambles}, we have
 \[||\mathcal{E}_2||=\min(\lambda_2(W_5),|V(W_5)|-\alpha(W_5)),\]
 where \(\lambda_2(G)\) denotes the smallest number of edges necessary to disconnect a graph \(G\) into two connected components, each with at least \(2\) vertices; and \(\alpha(G)\) denotes the independence number of \(G\), the maximum size of a subset of \(V(G)\) with no two elements adjacent in \(G\).  Inspecting all connected subgraphs of \(W_5\) with \(2\) or \(3\) vertices, we find \(\lambda_2(W_5)=4\); and we find \(\alpha(W_5)=2\).  Thus we have \(\sn(W_5)\geq ||\mathcal{E}_2||=4\), so \(\sn(W_5)=4\).

\begin{figure}[hbt]
    \centering
    \includegraphics[scale=1]{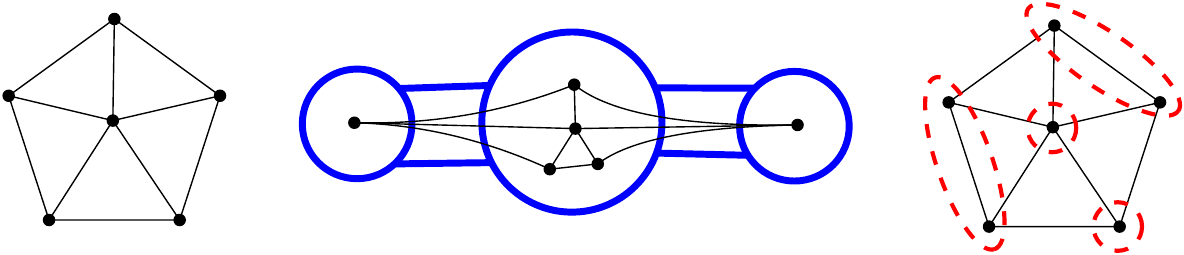}
    \caption{The wheel graph \(W_5\), with a tree-cut decomposition of width \(4\) and a disjoint scramble of order \(3\)}
    \label{figure:wheel_graph_scw}
\end{figure}

 We now claim that \(\dsn(W_5)=3\).  Let \(\mathcal{S}\) be a disjoint scramble on \(W_5\).  If \(\mathcal{S}\) has three eggs or fewer, \(||\mathcal{S}||\leq h(\mathcal{S})\leq 3\).  Otherwise, 
    \(\mathcal{S}\) has at least  four eggs.  Since there are six vertices in \(W_5\), at least two of the eggs must consist of a single vertex, meaning at least one of the eggs consists of a single vertex of degree \(3\).  The set of edges incident to that vertex then forms an egg-cut, so \(||\mathcal{S}||\leq e(\mathcal{S})\leq 3\).  Therefore if \(\mathcal{S}\) is a disjoint scramble, it has order at most \(3\), meaning that \(\dsn(W_5)\leq 3\).  There do indeed disjoint scrambles of order \(3\), such as the one on the right in Figure \ref{figure:wheel_graph_scw}, so \(\dsn(W_5)=3\).
 \end{example}

 In the following proposition we demonstrate precisely when it is possible for scramble number and disjoint scramble number to be distinct.
 
\begin{proposition}
For all graphs $G$ with $\sn(G) \leq 3$, we have $\dsn(G)=\sn(G)$.  However, for \(n\geq 4\), there exists a graph \(G\) with \(n=\sn(G)>\dsn(G)\).
\end{proposition}

\begin{proof}
If $G$ is a graph with $\sn(G)=1$, then letting \(v\in V(G)\) the scramble $\s=\{\{v\}\}$ has $||\s||=1$ and so $\dsn(G)=\sn(G)$.  If $G$ is a graph with $\sn(G)=2$, $G$ is not a tree.  Thus, $G$ contains some cycle \(C\).  Let \(u\) and \(v\) be two distinct vertices on \(C\), and let $\s=\{\{u\},\{v\}\}$.  An egg-cut must cut $C$ in two places, and so $e(\s)\geq 2$.  Additionally, as the two eggs of $\s$ are disjoint, $h(\s)=2$ and thus $||\s||=2$.  It follows that $\dsn(G)\geq 2$ and so $\sn(G)=\dsn(G)$.  

If $\sn(G)=3$, then by Theorem \ref{theorem:main_theorem} we know that $G$ contains one of the graphs from Figure \ref{figure:four_forbidden} as a topological minor; call that graph \(M\).  Let \(H\) be a subgraph of \(G\) that is a subdivision of \(M\).  Build a scramble on \(H\) as follows:  start with the scramble \(\mathcal{S}_M\) on \(M\) illustrated in Figure \ref{figure:four_forbidden_scrambles}, and modify the eggs to include any vertices in edges that are subdivided in passing from \(M\) to \(H\), thereby obtaining a scramble \(\mathcal{S}_H\) on \(H\).  The eggs have remained disjoint, and the number of edge-disjoint paths connecting each pairs of eggs has remained unchanged, so \(||\mathcal{S}_H||=||\mathcal{S}_M||=3\).  Finally, we may consider our scramble as a scramble \(\mathcal{S}_G\) on \(G\), and we have \(\dsn(G)\geq ||\mathcal{S}_G||\geq ||\mathcal{S}_H||=3\) (in passing from \(H\) to \(G\), the hitting number remains unchanged, and the egg-cut number could only increase).  Since \(\sn(G)=3\), we have \(\dsn(G)=3\), as desired.

\begin{figure}[hbt]
    \centering
    \includegraphics[scale=1]{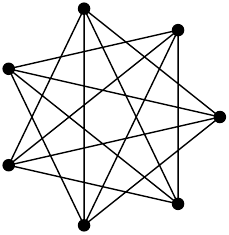}
    \caption{A graph with scramble number \(5\) and disjoint scramble number \(4\)}
    \label{figure:k7_minus_cycle}
\end{figure}

We already know that there exists a graph \(G\) with \(\dsn(G)<\sn(G)=4\), namely the wheel graph \(W_5\) from Example \ref{example:w5}.  Now let \(n\geq 5\), and let \(G\) be the graph obtained from a complete graph \(K_{n+2}\) by deleting a cycle of length \(n+2\); this graph is illustrated for \(n=5\) in Figure \ref{figure:k7_minus_cycle}.  Note that \(\alpha(G)=2\): every vertex is connected to all but two other vertices, which are connected to each other.   The minimum degree \(\delta(G)\) equals \(n-1\), which since \(n\geq 5\) is at least \(\lfloor (n+2)/2\rfloor-1= \lfloor |V(G)|/2\rfloor -1\). This allows us to apply \cite[Corollary 3.2]{Ech} to deduce that \(\sn(G)=|V(G)|-\alpha(G)=(n+2)-2=n\).  To see that \(\dsn(G)<n\), suppose that \(\mathcal{S}\) is a disjoint scramble on \(G\).  If \(h(\mathcal{S})\geq n\), then at least one of the \(n\) or more eggs has at most one vertex in it.  But deleting the \(n-1\) edges incident to that vertex would then give an egg-cut, meaning \(||\mathcal{S}||\leq n-1\).  Thus \(\dsn(G)\leq n-1<n=\sn(G)\).
%Additionally, $\lambda(G)\geq 3$ for graphs $A,B,$ and $C$, so we have that $e(\s_i)\geq 3$ and so $\dsn(G) \geq 3$ for the graphs $A,B,$ and $C$. In graph $D$, to disconnect any two eggs we must use one of the doubled edges along with an additional edge to cleave the other side of the cycle, so we have $e(\s_D)=3$ and thus $\dsn(D) \geq 3$.  So, if $G$ is a graph with $\sn(G)=3$, $G$ has a topological minor $A, B, C$ or $D$ that has a disjoint scramble of order three.   Thus, $\dsn(G)=\sn(G)$ when $\sn(G)=3$.  
\end{proof}

We close with a result concerning the computational complexity of scramble number.  It is known that computing the scramble number of a graph is NP-hard \cite{Ech}. We can ask if, for a fixed \(k\), there exists a polynomial-time algorithm to check whether \(\sn(G)\leq k\). For \(k=1\), the answer is yes, since there are efficient algorithms to check if a graph is a tree or not.  For \(k=2\), Theorem \ref{theorem:main_theorem} allows us to determine the answer is still yes.  We first recall the following useful result.  %It turns out for \(\dsn(G)\), there is.

%\begin{proposition}
%For any fixed positive integer \(k\), there is a polynomial-time algorithm to determine whether \(\dsn(G)\leq k\).
%\end{proposition}

%\begin{proof} First we note \(G\) has a disjoint scramble \(\s\) wi
%\end{proof}

%However, now that we have a characterization of scramble number two, we may be able to check if a graph $G$ has $\sn(G)\leq 2$ in a relatively short amount of time. We will make use of the following result. %In fact, it is the case that there is an $O(|V(G)|^3)$ time algorithm which determines if $G$ has $\sn(G)=2$.

\begin{theorem}[Theorem 1.1 in \cite{Gro}]\label{np} For every graph $H$, there is a $O(|V(G)|^3)$ time algorithm that decides if $H$ is a topological minor of G.

\end{theorem}

\begin{corollary}
There is a $O(|V(G)|^3)$ time algorithm that decides whether $\sn(G)\leq 2$.
\end{corollary}

\begin{proof}
Let $G$ be a graph.  From Theorem \ref{np} we know there is a $O(|V(G)|^3)$ time algorithm to check if any fixed graph $H$ is a topological minor of $G$.  To check if $G$ has $\sn(G)\leq 2$ we must simply run this algorithm four times to check for graphs \(K_4\), \(P_{3;3}\), \(C_{3;2,2,1}\), and \(LL_6\). If $G$ contains one of these graphs as a topological minor, then $\sn(G)>2$.  If $G$ does not, then $\sn(G) \leq 2$.  %Thus, there exists a $O(|V(G)|^3)$ time algorithm to check if $G$ has $\sn(G)>2$. 
\end{proof}

For fixed \(k\geq 3\), the question is to our knowledge open.

\begin{question} For what fixed values of \(k\geq 3\) does there exist a polynomial time algorithm to determine whether \(\sn(G)\leq k\)?
\end{question}

We remark that for \(k=3\), the answer will be the same if we replace \(\sn(G)\) with \(\dsn(G)\) (which may be simpler to work with).

\bibliographystyle{plain}

\end{document}